\documentclass{amsart}

\usepackage{amsthm,amssymb}
\usepackage[pagebackref,colorlinks,linkcolor=red,citecolor=blue,urlcolor=blue,hypertexnames=false]{hyperref}
\usepackage{amsrefs}
\usepackage{amscd}

\input xypic
\xyoption{all}
\newdir{ >}{{}*!/-5pt/@{>}}

\renewcommand{\top}{\rm top}
\newcommand{\Tr}{\mathrm{Tr}}
\newcommand{\rel}{\rm rel}
\newcommand{\reg}{\mathrm{reg}}
\newcommand{\dif}{\rm diff}

\newcommand{\CT}{\kappa}
\newcommand{\Gal}{\mathrm{Gal}}
\newcommand{\con}{\mathrm{con}}

\def\cB{\mathcal B}

\def\cE{\mathcal E}

\def\cF{\mathcal F}
\def\fA{\mathfrak A}

\def\cG{\mathcal G}

\def\cL{\mathcal L}

\def\cO{\mathcal O}

\def\nil{\mathrm{nil}}

\def\hotimes{\hat{\otimes}}
\newcommand{\C}{\mathbb{C}}
\def\cotimes{\otimes_\C }

\newcommand{\Q}{\mathbb{Q}}
\newcommand{\R}{\mathbb{R}}

\newcommand{\fin}{\mathcal{F}in}

\newcommand{\fcyc}{\mathcal{F}cyc}
\newcommand{\cyc}{\mathcal{C}yc}

\newcommand{\Z}{\mathbb{Z}}

\newcommand{\hofi}{\mathrm{hofiber}}
\newcommand{\hocofi}{\mathrm{hocofiber}}

\newcommand{\org}{\mathrm{Or}G}

\def\bu{\bullet}

\def\colim{\operatornamewithlimits{colim}}

\def\map{\operatorname{map}}
\def\lra{\longrightarrow}

\def\iso{\stackrel{\cong}\lra}
\def\triqui{\vartriangleleft}
\def\weq{\overset\sim\lra}

\def\ct{\tau}
\def\mk{\mu}
\numberwithin{equation}{section}
\theoremstyle{plain}
\newtheorem{thm}[equation]{Theorem}
\newtheorem{coro}[equation]{Corollary}
\newtheorem{lem}[equation]{Lemma}
\newtheorem{prop}[equation]{Proposition}

\newcommand{\comment}[1]{}  

\theoremstyle{definition}

\theoremstyle{remark}
\newtheorem{rem}[equation]{Remark}
\newtheorem{ex}[equation]{Example}

\begin{document}

\title{Trace class operators, regulators, and assembly maps in $K$-theory}
\author{Guillermo Corti\~nas}
\email{gcorti@dm.uba.ar}\urladdr{http://mate.dm.uba.ar/\~{}gcorti}
\author{Gisela Tartaglia}
\email{gtartag@dm.uba.ar}
\address{Dep. Matem\'atica-IMAS, FCEyN-UBA\\ Ciudad Universitaria Pab 1\\
1428 Buenos Aires\\ Argentina}
\thanks{The first author was partially supported by MTM2012-36917-C03-02. Both authors
were supported by CONICET, and partially supported by
grants UBACyT 20020100100386 and PIP 112-200801-00900.}

\maketitle

\section{Introduction}
Let $G$ be a group, $\fin$ the family of its finite subgroups, and $\cE(G,\fin)$ the classifying space. Let $\cL^1$ be the algebra of trace-class operators
in an infinite dimensional, separable Hilbert space over the complex numbers. Consider the rational
assembly map in homotopy algebraic $K$-theory
\begin{equation}\label{intro:asskh}
H_p^G(\cE(G,\fin),KH(\cL^1))\otimes\Q\to KH_p(\cL^1[G])\otimes\Q. 
\end{equation}
The rational $KH$-isomorphism conjecture (\cite{balukh}*{Conjecture 7.3}) predicts that \eqref{intro:asskh} is
an isomorphism; it follows from a theorem of Yu (\cite{yu}, \cite{cortar}) that it is always injective. In the current
article we prove the following. 

\begin{thm}\label{thm:intromain}
Assume that \eqref{intro:asskh} is surjective. Let $n\equiv p+1\mod 2$.
Then:
\item[i)] The rational assembly map for the trivial family 
\begin{equation}\label{intro:noviz}
H_{n}^G(\cE(G,\{1\}),K(\Z))\otimes\Q\to K_{n}(\Z[G])\otimes\Q
\end{equation}
is injective.
\item[ii)] For every number field $F$, the rational assembly map
\begin{equation}\label{intro:assF}
H_{n}^G(\cE(G,\fin),K(F))\otimes\Q\to K_{n}(F[G])\otimes\Q 
\end{equation}
is injective. 
\end{thm}

We remark that the $K$-theory Novikov conjecture asserts that part i) of the theorem above holds for all $G$, and that part ii) is equivalent to the rational injectivity part of the $K$-theory Farrell-Jones conjecture for number fields (\cite{LR2}*{Conjectures 1.51 and 2.2 and Proposition 2.14}).  

The idea of the proof of Theorem \ref{thm:intromain} is to use an algebraic, equivariant version of Karoubi's multiplicative $K$-theory.
The latter theory assigns groups $MK_n(\fA)$ $(n\ge 1)$ to any unital Banach algebra $\fA$, which fit into a long
exact sequence 
\[
HC^{\top}_{n-1}(\fA)\to MK_n(\fA)\to K_n^{\top}(\fA)\overset{Sch^{\top}_{n}}\to HC^{\top}_{n-2}(\fA).    
\]
Here $HC^{\top}$ is the cyclic homology of the completed cyclic module $C^{\top}_n(\fA)=\fA\hotimes\dots\hotimes\fA$ ($n+1$ factors), 
$ch^{\top}$ is the Connes-Karoubi Chern character with values in its periodic cyclic homology $HP_*^{\top}(\fA)$, and $S$ is the
periodicity operator. Karoubi introduced a multiplicative Chern character
\begin{equation}\label{intro:mu}
\mu_n:K_n(\fA)\to MK_n(\fA).
\end{equation}
In particular if $\cO$ is the ring of integers in a number field $F$ one can consider the composite  
\begin{equation}\label{intro:reg}
K_n(\cO)\to K_n(\C)^{\hom(F,\C)}\to MK_n(\C)^{\hom(F,\C)}.
\end{equation}
By comparing this map with the Borel regulator, Karoubi showed in \cite{karast} that \eqref{intro:reg} is rationally injective.
It follows that 
\begin{equation}\label{intro:zreg}
K_n(\Z)\to MK_n(\C)
\end{equation}
is rationally injective. In the current paper we assign, to every unital $\C$-algebra $A$, groups $\CT_n(A)$ $(n\in \Z)$ which fit into a long
exact sequence  
\[
HC_{n-1}(A/\C)\to \CT_n(A)\to KH_n(\cL^1\cotimes A )\overset{{\Tr} Sch_n}\to HC_{n-2}(A/\C).
\]
Here $HC(/\C)$ is algebraic cyclic homology of $\C$-algebras, $ch$ is the algebraic Connes-Karoubi Chern character and $\Tr$ is induced
by the operator trace. We also introduce a character
\begin{equation}\label{intro:ct}
\ct_n:K_n(A)\to \CT_n(A)\ \ (n\in\Z).
\end{equation}
If $\fA$ is a finite dimensional Banach algebra and  $n\ge 1$ then
$\CT_n(\fA)=MK_n(\fA)$ and \eqref{intro:mu} identifies with \eqref{intro:ct} (Proposition \ref{prop:mk=ct}). Both $\CT$ and $\tau$ have equivariant
versions, so that if $X$ is a $G$-space and $A$ is a $\C$-algebra, we have an assembly map
\[
H^G_n(X,\CT(A))\to \CT_n(A[G]).
\]
Let $\fcyc$ be the family of finite cyclic subgroups. We show
in Proposition \ref{prop:compuhgctc} that the map
\begin{equation*}\label{intro:fcycfin}
H_n^G(\cE(G,\fcyc),\CT(\C))\to H_n^G(\cE(G,\fin),\CT(\C))
\end{equation*}
is an isomorphism, and compute $H_n^G(\cE(G,\fcyc),\CT(\C))\otimes\Q$ in terms of the finite cyclic subgroups of $G$.
We use this and the rational injectivity of \eqref{intro:zreg} to show, in Proposition \ref{prop:monoequiregZ}, that
the map
\begin{multline}\label{intro:assnovi}
H_n^G(\cE(G,\{1\}),K(\Z))\to H_n^G(\cE(G,\fin),K(\C))
\overset{\tau}\to H_n^G(\cE(G,\fin),\CT(\C))
\end{multline}
is rationally injective. It is well-known \cite{LR2}*{Proposition 2.20} that the map
\[
H^G_n(\cE(G,\fcyc),K(R))\otimes\Q\to H^G_n(\cE(G,\fin),K(R))\otimes\Q
\]
is an isomorphism for every unital ring $R$. In particular, we may substitute $\fcyc$ for $\fin$ in \eqref{intro:assF}. We use this together with Proposition \ref{prop:compuhgctc} and the rational injectivity of
\[K_n(F)\to K_n(\C)^{\hom(F,\C)}\to MK_n(\C)^{\hom(F,\C)}\] (see Remark \ref{rem:kareg}), to show in Proposition \ref{prop:monoequiregF} that if $m\ge 1$, $\cyc_m$ is the family of cyclic subgroups whose order divides $m$, and $\zeta_m$ is a primitive $m$-root of
$1$, then the composite
\begin{equation}\label{intro:assf}
\xymatrix{H_n^G(\cE(G,\cyc_m), K(F))\otimes\Q\ar[r]\ar@{.>}[dr]& H_n^G(\cE(G,\cyc_m),K(\C))^{\hom(F(\zeta_m),\C)}\otimes\Q\ar[d]^\tau\\
 &H_n^G(\cE(G,\fcyc),\CT(\C))^{\hom(F(\zeta_m),\C)}\otimes\Q}
\end{equation}
is injective. Since the map $\colim_m\cE(G,\cyc_m)\to \cE(G,\fcyc)$ is an equivalence, it follows that if the rational assembly map
\begin{equation}\label{intro:assct}
H_n^G(\cE(G,\fcyc),\CT(\C))\otimes\Q\to \CT_n(\C[G])\otimes\Q 
\end{equation}
is injective then so are both \eqref{intro:noviz} and \eqref{intro:assF}. We show in 
Corollary \ref{coro:fjct} that if \eqref{intro:asskh} is surjective, then \eqref{intro:assct} is injective for $n\equiv p+1\mod 2$.
This proves Theorem \ref{thm:intromain}. 

The rest of this paper is organized as follows. In Section \ref{sec:prelis} we define $\CT_n(A)$ and the map $\ct_n:K_n(A)\to\CT_n(A)$. By definition,
if $n\le 0$, then $\CT_n(A)=KH_n(A\cotimes\cL^1)$ and $\tau_n$ is the identity map \eqref{ctlow}. We show in Proposition
\ref{prop:mk=ct} that if $n\ge 1$ and $\fA$ is a finite dimensional Banach algebra, then $\CT_n(\fA)=MK_n(\fA)$ and $\ct_n=\mu_n$. Karoubi's regulators and his injectivity results are recalled in Theorem \ref{thm:kar}. We use Karoubi's theorem to prove, in Lemma \ref{lem:regcyc}, that if $F$ is a number field, $C$ a cyclic group of order $m$, $n$ a multiple of $m$, and $\zeta_n$ a primitive $n$-root of $1$, then the composite
\[
K_*(F[C])\to K_*(F(\zeta_n)[C])\to
 K_*(\C[C])^{\hom(F(\zeta_n),\C)}\overset{\mu}\to MK_*(\C[C])^{\hom(F(\zeta_n),\C)}
\]
is rationally injective. The main result of Section \ref{sec:compuhgctc} is Proposition \ref{prop:compuhgctc}, which computes $H_n^G(\cE(G,\fcyc),\CT(\C))\otimes\Q$ in terms of group homology and
of the groups $\CT_*(\C[C])$ for $C\in\fcyc$. The resulting formula is similar to existing formulas for equivariant $K$ and cyclic homology, which are used in its proof (\cite{corel}, \cite{luch}, \cite{lurela}, \cite{LR2},\cite{LR1}). In Section \ref{sec:equi} we show that the rational $\CT(\C)$-assembly map is injective whenever the rational $KH(\cL^1)$-assembly map is surjective (Corollary \ref{coro:fjct}). For this we use the fact that for every $m\ge 1$, the assembly map 
\[
H_*^G(\cE(G,\cyc_m),HC(\C/\C))\to HC_*(\C[G])
\]
has a natural left inverse $\pi_m$, which makes the following diagram commute
\[
\xymatrix{
H_*^G(\cE(G,\cyc_m),KH(\cL^1))\otimes\Q \ar[r]\ar[d]^{\Tr Sch}& KH_*(\cL^1[G])\otimes\Q\ar[d]^{{\Tr}Sch}\\
H_{*-2}^G(\cE(G,\cyc_m),HC(\C/\C))& HC_{*-2}(\C[G])\ar[l]_(0.3){\pi_m} .}
\]
Hence for every $n$ we have an inclusion
\begin{equation}\label{intro:inc}
\Tr Sch(H_{n+1}^G(\cE(G,\cyc_m),KH(\cL^1))\otimes\Q)\subset\pi_m{\Tr}Sch(KH_{n+1}(\cL^1[G])\otimes\Q).
\end{equation}
We show in Proposition \ref{prop:equiva} that the rational assembly map
\[
H_n^G(\cE(G,\cyc_m),\CT(\C))\otimes\Q\to \CT_n(\C[G])\otimes\Q
\]
is injective if and only if the inclusion \eqref{intro:inc} is an equality. Corollary \ref{coro:fjct} is immediate
from this. Section \ref{sec:regeq} is concerned with proving that \eqref{intro:assnovi} and \eqref{intro:assf} are injective (Propositions \ref{prop:monoequiregZ}
and \ref{prop:monoequiregF}). Finally in Section \ref{sec:compa} we show that if the identity holds in \eqref{intro:inc} for $m=1$ then \eqref{intro:noviz} is injective (Theorem \ref{thm:forz}) and that if it holds for $m$, then
\[
H_n^G(\cE(G,\cyc_m),K(F))\otimes\Q\to K_n(F[G])\otimes\Q
\]
is injective for every number field $F$ (Theorem \ref{thm:forf}).

\goodbreak
 
\section{The character \texorpdfstring{$\ct:K(A)\to \CT(A)$}{}}\label{sec:prelis}
\numberwithin{equation}{subsection}
\subsection{Definition of \texorpdfstring{$\ct$}{ct}}
Let $A$ be a $\C$-algebra and $k\subset \C$ a subfield. Write $C(A/k)$ for Connes' cyclic module and $HH(A/k)$, $HC(A/k)$, $HN(A/k)$ and $HP(A/k)$ for the associated Hochschild, cyclic, negative cyclic and periodic cyclic chain complexes. When $k=\Q$ we omit it from our notation; thus for example, $HH(A)=HH(A/\Q)$. As usual, we write $S$, $B$ and $I$ for the maps appearing in Connes' $SBI$ sequence. To simplify notation we shall make no distinction between a chain complex and the spectrum the Dold-Kan corresponce associates to it. We write $KH$ for Weibel's homotopy algebraic $K$-theory and $K^{\nil}$
for the fiber of the comparison map $K\to KH$. We have a map of fibration sequences \cite{friendly}*{\S11.3}
\[
 \xymatrix{K^{\nil}(A)\ar[r]\ar[d]_\nu &K(A)\ar[r]\ar[d]& KH(A)\ar[d]^{ch}\\
           HC(A)[-1]\ar[r]_B& HN(A)\ar[r]_I& HP(A).}
\]
Here $ch$ is the Connes-Karoubi character. Write $\cB$ for the algebra of bounded operators in an infinite dimensional, separable Hilbert space, and $\cL^1\triqui\cB$ for the ideal of trace class operators. Recall from \cite{cq} that $HP$ satisfies excision; in particular, the canonical map $HP(\cL^1\cotimes  A)\to HP(\cB\cotimes  A:\cL^1\cotimes A)$ is a quasi-isomorphism. We shall abuse notation and write $Sch$ for the map that makes the following diagram commute
\begin{equation}\label{map:sck}
\xymatrix{KH(\cL^1\cotimes  A)[+1]\ar[d]_{Sch}\ar[r]^{ch}& HP(\cL^1\cotimes  A)[+1]\ar[d]^\wr \\
HC(\cB\cotimes A:\cL^1\cotimes A)[-1]&HP(\cB\cotimes A:\cL^1\cotimes  A)[+1]\ar[l]^S.}
\end{equation}
By \cite{CT}*{Theorems 6.5.3 and 7.1.1}, the map $\nu:K^{nil}(\cB\cotimes A:\cL^1\cotimes A)\to HC(\cB\cotimes A:\cL^1\cotimes A)[-1]$ is an equivalence, and thus
the map $Sch$ fits into a fibration sequence
\begin{equation*}
KH(\cL^1\cotimes A)[+1]\overset{Sch}\longrightarrow HC(\cB\cotimes A:\cL^1\cotimes A)[-1]\to K(\cB\cotimes A:\cL^1\cotimes A).
\end{equation*}
On the other hand the operator trace $\Tr:\cL^1\to \C$ induces a map of cyclic modules
\begin{gather}\label{map:taucyc}
\Tr:C(\cB\cotimes A:\cL^1\cotimes A)\to C(A/\C)\\
\Tr(b_0\otimes a_0\otimes\dots\otimes b_n\otimes a_n)=\Tr(b_0\cdots b_n)a_0\otimes\dots\otimes a_n.\nonumber
\end{gather}
Note that $\Tr$ is defined on $b_0\cdots b_n$ since at least one of the $b_i$ is in $\cL^1$. In particular $\Tr$ induces a chain map
\begin{equation}\label{map:utauC}
\Tr:HC(\cB\cotimes A:\cL^1\cotimes A)\to HC(A/\C).
\end{equation}
We define $\CT(A)$ as the homotopy cofiber of the composite of \eqref{map:utauC} and the map $Sch$ of \eqref{map:sck}
\[
\CT(A):=\hocofi(KH(\cL^1\cotimes A)[+1]\overset{\Tr Sch}\longrightarrow HC(A/\C)[-1]).
\]
Thus because, by definition, cyclic homology vanishes in negative degrees, we have
\begin{equation}\label{ctlow}
\CT_n(A)=KH_n(\cL^1\cotimes A) \quad (n\le 0).
\end{equation}

By construction, there is an induced map $K(\cB\cotimes A:\cL^1\cotimes  A)\to \CT(A)$ which fits into a commutative diagram
\begin{equation}\label{diag:deftau}
\xymatrix{
KH(\cL^1\cotimes  A)[+1]\ar@{=}[d]\ar[r]^(0.451){Sch}& HC(\cB\cotimes A:\cL^1\cotimes  A)[-1]\ar[d]\ar[r]&K(\cB\cotimes A\negthickspace:\negthickspace\cL^1\cotimes  A)\ar[d]\\
KH(\cL^1\cotimes  A)[+1]\ar[r]&HC(A/\C)[-1]\ar[r]& \CT(A)}
\end{equation}
A choice of a rank one projection $p$ gives a map $A\to \cL^1\cotimes A$, $a\mapsto p\otimes a$, and therefore a map $K(A)\to K(\cB\cotimes A:\cL^1\cotimes  A)$. We shall be interested in the composite
\begin{equation}\label{map:ct}
\ct:K(A)\to K(\cB\cotimes A :\cL^1\cotimes  A)\to \CT(A).
\end{equation}

\subsection{Comparison with Karoubi's multiplicative Chern character}\label{subsec:karmch}
Suppose now that $\fA$ is a unital Banach algebra. Let $\Delta_\bu^{\dif}\fA=C^\infty(\Delta_\bu,\fA)$ be the simplicial algebra of $\fA$-valued $C^\infty$-functions on the standard simplices. 
Write $KV^{\dif}(\fA)$ for the diagonal of the bisimplicial space $[n]\mapsto BGL(\Delta_n^{\dif}\fA)$. We have
\[
K_n^{\top}(\fA)=\pi_nKV^{\dif}(\fA)\ \ (n\ge 1).
\]
Consider the fiber $\cF(\fA)=\hofi(BGL^+(\fA)\to KV^{\dif}(\fA))$. We have a homotopy fibration
\begin{equation}\label{karfib}
\Omega BGL^+(\fA)\to \Omega KV^{\dif}(\fA)\to \cF(\fA).
\end{equation}
Let $\hotimes$ be the projective tensor product of Banach spaces and let $C^{\top}(\fA)$ be the cyclic module with $C^{\top}(\fA)_n=\fA\hotimes\dots\hotimes \fA$ ($n+1$ factors). Write $HC^{\top}(\fA)$ and $HP^{\top}(\fA)$ for the cyclic and periodic cyclic complexes of $C^{\top}(\fA)$. In \cite{karmult}
(see also \cite{karast}*{\S7}), Max Karoubi constructs a map $ch^{\rel}:\cF(\fA)\to HC^{\top}(\fA)[-1]$ and defines his multiplicative $K$-groups  as the homotopy groups 
\[
MK_n(\fA)=\pi_{n}(\hofi(KV^{\dif}(\fA)\to HC^{\top}(\fA)[-2])) \quad (n\ge 1).
\]
He further defines the multiplicative Chern character as the induced map $\mu_n:K_n(\fA)\to MK_n(\fA)$ $(n\ge 1)$. 

\begin{prop}\label{prop:mk=ct}
Let $\fA$ be a unital Banach algebra, and let $n\ge 1$. Then there is a natural map $\CT_n(\fA)\to MK_n(\fA)$ which makes the following diagram commute
\[
\xymatrix{K_n(\fA)\ar[r]^{\ct_n}\ar[dr]_{\mk_n}& \CT_n(\fA)\ar[d]\\
                                            & MK_n(\fA).}
\]
If furthermore $\fA$ is finite dimensional, then $\CT_n(\fA)\to MK_n(\fA)$ is an isomorphism.
\end{prop}

\begin{proof} Consider the simplicial ring 
\[
\Delta_\bu \fA:[n]\mapsto\Delta_n\fA=\fA[t_0,\dots,t_n]/\langle 1-\sum_{i=0}^nt_i\rangle.
\]
Let $KV(\fA)$ be the diagonal of the bisimplicial set $BGL(\Delta_\bu \fA)$. We have a homotopy commutative diagram
\begin{equation}\label{diag:ctvsmk}
\xymatrix{
KV(\cL^1\cotimes \fA)\ar[r]\ar[d]& KV^{\dif}(\cL^1\hotimes \fA)\ar[d]&KV^{\dif}(\fA)\ar[l]\ar[d]\\
HC(\cB\cotimes \fA :\cL^1\cotimes  \fA)[-2]\ar[r]\ar[d]^{\Tr}&HC^{\top}(\cB\hotimes \fA :\cL^1\hotimes\fA)[-2]\ar[d] & HC^{\top}(\fA)[-2]\ar[l]\ar@{=}[dl]\\
HC(\fA/\C)[-2]\ar[r]& HC^{\top}(\fA)[-2]}
\end{equation}
By \cite{CT}*{Lemma 3.2.1 and Theorem 6.5.3(ii)} and \cite{kh}*{Proposition 1.5} (or \cite{friendly}*{Proposition 5.2.3}), the natural map 
$$KV_n(\cL^1\cotimes \fA)\to KH_n(\cL^1\cotimes \fA)$$
is an isomorphism for $n\ge 1$. It follows from this that for $n\ge 1$, the group $\CT_n(\fA)$ is isomorphic to $\pi_n$ of the fiber of the composite of the first
column of diagram \eqref{diag:ctvsmk}. On the other hand, by Karoubi's density theorem, the map $KV^{\dif}(\fA)\to KV^{\dif}(\cL^1\hotimes \fA)$ is an equivalence;
inverting it and taking fibers and homotopy groups, we get a natural map $\CT_n(\fA)\to MK_n(\fA)$ $(n\ge 1)$. The commutativity of the diagram of the proposition
is clear. If now $\fA$ is finite dimensional, then $\fA\hotimes V=\fA\cotimes V$ for any locally convex vector space $V$. Hence the map $HC(\fA/\C)\to HC^{\top}(\fA)$ is
the identity map. Furthermore, by \cite{CT}*{Theorem 3.2.1}, the map $KV(\cL^1\cotimes \fA)\to KV^{\dif}(\cL^1\hotimes \fA)$ is an equivalence. It follows that
$\CT_n(\fA)\to MK_n(\fA)$ is an isomorphism for all $n\ge 1$, finishing the proof.
\end{proof}

\begin{ex}
We have
\[
 \CT_n(\C)=\left\{\begin{matrix} \C^* & n\ge 1, \text{ odd.}\\
                                \Z   & n\le 0, \text{ even.}\\
                                0    & \text{ otherwise.}
                 \end{matrix}
\right.
\]
\end{ex}

\subsection{Regulators}

In view of Proposition \ref{prop:mk=ct} above we may substitute $\tau$ for $\mu$ in the theorem below.

\begin{thm}\label{thm:kar}\cite{karast}*{Th\'eor\`eme 7.20}
Let $\cO$ be the ring of integers in a number field $F$. Write $F\otimes\R\cong\R^{r_1}\oplus \C^{r_2}$; put $r=r_1+r_2$. Then the inclusion $\cO\subset\C^r$ followed by the the map $\mk_n:K_n(\C)^r\to MK_n(\C)^r$ induces
a monomorphism $K_n(\cO)\otimes\Q\to MK_n(\C)^r\otimes\Q$ $(n\ge 1)$.
\end{thm}

\begin{rem}\label{rem:kareg}
It follows from classical results of Quillen that the map $K_n(\cO)\to K_n(F)$ is a rational isomorphism for $n\ge 2$. Thus $K_n(F)\to MK_n(\C)^r$ is rationally injective for $n\ge 2$. Moreover, $K_1(F)\to MK_1(\C)^r$ is injective too, since the map  $\tau_1:K_1(\C)\to MK_1(\C)$ is the identity of $\C^*$. Observe that the isomorphism $F\otimes\R\cong\R^{r_1}\oplus\C^{r_2}$ of Theorem \ref{thm:kar} is not canonical; it implies choosing $r_2$ nonreal embeddings $F\to \C$ out of the total $2r_2$, so that no two of them differ by complex conjugation. On the other hand the map 
\[
\iota:F\to \C^{\hom(F,\C)},\ \ \iota(x)_\sigma=\sigma(x) 
\] 
is canonical. Moreover, the composite 
\begin{equation}\label{map:buenreg}
\reg_n(F):K_n(F)\to K_n(\C)^{\hom(F,\C)}\to MK_n(\C)^{\hom(F,\C)}
\end{equation}
is still a rational monomorphism. Indeed the map of the theorem is obtained by composing \eqref{map:buenreg} with a projection $MK_n(\C)^{\hom(F,\C)}\to MK_n(\C)^r$.  
\end{rem}

\begin{lem}\label{lem:regcyc}
Let $F$ be a number field, $C$ a cyclic group of order $m$, $n$ a multiple of $m$ and $\zeta_n$ a primitive $n$-th root of $1$. Then the composite map
\begin{multline*}
K_*(F[C])\to K_*(F(\zeta_n)[C])\overset{\iota}\to\\
 K_*(\C[C])^{\hom(F(\zeta_n),\C)}\overset{\mu}\to MK_*(\C[C])^{\hom(F(\zeta_n),\C)}
\end{multline*}
is rationally injective. 
\end{lem}
\begin{proof}
Let $G_m=\Gal(F(\zeta_m)/F)$; if $M$ is a $G_m$-module, write $M^{G_m}$ for the fixed points. By \cite{luch}*{Lemma 8.4}, the map $F[C]\to F(\zeta_m)[C]$ induces an isomorphism $K_*(F[C])\otimes\Q\to K_*(F(\zeta_m)[C])^{G_m}\otimes\Q$. In particular, $K_*(F[C])\to K_*(F(\zeta_m)[C])$ is rationally injective. Now if $\sigma:F(\zeta_m)\to E$ is a field homomorphism, then $K_*(E[C])=K_*(E)^m$, and the map 
$K_*(F(\zeta_m)[C])\to K_*(E[C])$ decomposes into a direct sum of $m$ copies of the map $K_*(F(\zeta_m))\to K_*(E)$. In particular this applies when $E\in\{F(\zeta_n),\C\}$. 
In view of Theorem
\ref{thm:kar} and Remark \ref{rem:kareg}, it follows that both $K_*(F(\zeta_m)[C])\to MK_*(\C[C])^{\hom(F(\zeta_m),\C)}$ and $K_*(F(\zeta_n)[C])\to MK_*(\C[C])^{\hom(F(\zeta_n),\C)}$ are rationally injective. Summing up,
we have a commutative diagram
\[
\xymatrix{K_*(F[C])\ar[d]&\\
K_*(F(\zeta_m)[C])\ar[r]\ar[d]& MK_*(\C[C])^{\hom(F(\zeta_m),\C)}\ar[d]\\
K_*(F(\zeta_n)[C])\ar[r]& MK_*(\C[C])^{\hom(F(\zeta_n),\C)}
}
\]
We have shown that the first vertical map on the left and the two horizontal maps are rationally injective. Since the vertical map on the right is injective, we conclude that the composite of the left
column followed by the bottom horizontal arrow is a rational monomorphism, finishing the proof. 
\end{proof}

\section{Rational computation of equivariant \texorpdfstring{$\CT$}{CT}-homology}\label{sec:compuhgctc}
\numberwithin{equation}{section}

Let $G$ be a group and let $\org$ be its orbit category. For $G/H\in\org$, let $\cG(G/H)=\cG^G(G/H)$ be the transport groupoid.
It follows from \cite{cortar}*{\S3} that the diagram \eqref{diag:deftau} can be promoted to a commutative diagram of $\org$-spectra whose columns are homotopy fibrations 
\[
\xymatrix{
KH(\cL^1\cotimes  A[\cG(G/H)])[+1]\ar[d]^{Sch}\ar[dr]^{\Tr Sch}&\\
HC(\cB\cotimes A[\cG(G/H)]:\cL^1\cotimes  A[\cG(G/H)])[-1]\ar[r]\ar[d]&HC(A[\cG(G/H)]/\C)[-1]\ar[d]\\ 
K(\cB\cotimes A[\cG(G/H)] :\cL^1\cotimes  A[\cG(G/H)])\ar[r]  &\CT(A[\cG(G/H)])}
\]
If now $X$ is any $G$-simplicial set, then taking $G$-equivariant homology yields a diagram whose columns are again homotopy fibrations
\begin{equation}\label{diag:equi}
\xymatrix{
H^G(X,KH(\cL^1\cotimes  A))[+1]\ar@{=}[r]\ar[d]^{Sch}& H^G(X,KH(\cL^1\cotimes  A))[+1]\ar[d]\\
H^G(X,HC(\cB\cotimes A:\cL^1\cotimes  A))[-1]\ar[r]^-{\Tr}\ar[d]&H^G(X,HC(A/\C))[-1]\ar[d]\\ 
H^G(X, K(\cB\cotimes A:\cL^1\cotimes  A))\ar[r]  &H^G(X,\CT(A))}
\end{equation}

Hence 
\begin{equation}\label{hgcofi}
H^G(X,\CT(A))=\hocofi(H^G(X,KH(\cL^1\cotimes A))[+1]\to H^G(X,HC(A/\C))[-1]). 
\end{equation}
Similarly, a choice of rank one projection induces a map of $\org$-spectra 
\[
K(A[\cG(G/H)])\to K(\cB\cotimes A[\cG(G/H)]:\cL^1\cotimes A[\cG(G/H)]). 
\]
Taking equivariant homology we obtain a map
\[
H^G(X,K(A))\to H^G(X,K(\cB\cotimes A:\cL^1\cotimes A)).
\]
Composing this map with the bottom arrow in diagram \eqref{diag:equi} we obtain an equivariant character
\begin{equation}\label{map:equitau}
\tau:H^G(X,K(A))\to H^G(X,\CT(A)). 
\end{equation}
In what follows we shall be interested in several families of finite subgroups of a given group. We write $\fin$ and $\fcyc$ for the family of finite 
subgroups and the subfamily of those
finite subgroups that are cyclic. If $m\ge 1$ we write $\cyc_m$ for the family of those cyclic subgroups whose order divides $m$. 
If $G$ is a group and $\cF$ a family of subgroups, we write $\cE(G,\cF)$
for the corresponding classifying space. If $H\subset G$ is a subgroup in the family $\cF$, we write $(H)$ for the conjugacy class of $H$ and  
\[
(\cF)=\{(H):H\in\cF\}
\]
for the set of all conjugacy classes of subgroups of $G$ in the family $\cF$.
If $G$ is a group and $C\subset G$ is a cyclic subgroup, we write $N_GC$ for its normalizer, $Z_GC$ for its centralizer, and put
\[
W_GC=N_GC/Z_GC.
\]
If $C\subset G$ is a finite cyclic group, and $A(C)$ is its Burnside ring, then there is a canonical isomorphism $A(C)\otimes\Q\cong \Q^{\fin C}$. We write $\theta_C\in A(C)\otimes\Q$ for the element corresponding to the characteristic function $\chi_C\in \Q^{\fin C}$.

\begin{prop}\label{prop:compuhgctc}
Let $G$ be a group. Then the map $H_*^G(\cE(G,\fcyc),\CT(\C))\to H_*^G(\cE(G,\fin),\CT(\C))$ is an isomorphism and
\begin{multline}\label{hctcrat}
H_n^G(\cE(G,\fcyc),\CT(\C))\otimes\Q=\\
\bigoplus_{p+q=n}\bigoplus_{(C)\in(\fcyc)}H_{p}(Z_GC,\Q)\otimes_{\Q[W_GC]}\theta_C\cdot \CT_q(\C[C])\otimes\Q
\end{multline}
\end{prop}
\begin{proof}
If $H$ is a finite subgroup, then the equivalence $KH(\cL^1)\weq K^{\top}(\cL^1)\overset{\sim}\leftarrow K^{\top}(\C)$
induces an equivalence $KH(\cL^1[\cG(G/H)])\overset{\sim}\to K^{\top}(C^*(\cG(G/H)))$. Hence if $X$ is a $(G,\fin)$-complex, we have an equivalence 
\begin{equation}\label{map:khkto}
H^G(X,KH(\cL^1))\weq H^G(X,K^{\top}(\C)).
\end{equation}
Thus $H^G(\cE(G,\fcyc),KH(\cL^1))\to H^G(\cE(G,\fin),KH(\cL^1))$ is an equivalence because $H^G(\cE(G,\fcyc),K^{\top}(\C))\to H^G(\cE(G,\fin),K^{\top}(\C))$ is (\cite{LR2}*{Proposition 2.13}). Similarly, $H^G(\cE(G,\fcyc),HC(\C/\C))\to H^G(\cE(G,\fin),HC(\C/\C))$ is an equivalence (see \cite{LR1}*{\S9} or \cite{corel}*{\S7}). From \eqref{hgcofi} and what we have just proved, it follows that
$H^G(\cE(G,\fcyc),\CT(\C))\to H^G(\cE(G,\fin),\CT(\C))$ is an equivalence. This shows the first assertion of the proposition. From \eqref{map:khkto}, 
\cite{lurela}*{Theorem 0.7} and 
\cite{LR2}*{Theorem 8.4}, we get
\begin{multline*}
H_n^G(\cE(G,\fcyc),KH(\cL^1))\otimes\Q
=\\
\bigoplus_{p+q=n}\bigoplus_{(C)\in(\fcyc)}H_p(Z_GC,\Q)\otimes_{\Q[W_GC]}\theta_C\cdot K^{\top}_q(\C[C])\otimes\Q.
\end{multline*}
Next write $\con_f(G)$ for the conjugacy classes of elements of $G$ of finite order, and $Gen(C)$ for the set of all generators of $C\in\fcyc$. By using \cite{LR1}*{Lemma 7.4} and the argument of the proof of \cite{cortar}*{Proposition 2.2.1} we obtain
\begin{gather*}
H^G_n(\cE(G,\fcyc),HC(\C/\C))=\\
=\bigoplus_{p+q=n}\bigoplus_{(g)\in\con_f(G)}H_p(Z_G(\langle g\rangle),\Q)\otimes HC_q(\C/\C)\\
                          =\bigoplus_{p+q=n}\bigoplus_{(C)\in (\fcyc)}H_p(Z_GC,\Q)\otimes_{\Q[W_G(C)]}\map(Gen(C),HC_q(\C/\C))\\
                          =\bigoplus_{p+q=n}\bigoplus_{(C)\in (\fcyc)}H_p(Z_GC,\Q)\otimes_{\Q[W_G(C)]}\theta_C\cdot HC_q(\C[C]/\C)
\end{gather*} 
It follows from the proof of Proposition \ref{prop:mk=ct} that under the isomorphism \eqref{map:khkto} and the identity $H^G_*(-,HC(\C/\C))=H_*^G(-,HC^{\top}(\C))$ the map ${\Tr}Sch$ identifies
with ${\Tr} Sch^{\top}$. Hence, by naturality, the map 
$$({\Tr} Sch)_n:H_{n+1}^G(\cE(G,\fcyc),KH(\cL^1))\to H_{n-1}^G(\cE(G,\fcyc),HC(\C/\C))$$ 
is induced by the maps ${\Tr}Sch^{\top}_q:K^{\top}_{q+1}(\C[C])\to HC^{\top}_{q-1}(\C[C])=HC_{q-1}(\C[C]/\C)$. The computation of $H_n^G(\cE(G,\fin),\CT(\C))\otimes\Q$ is now
immediate from this.
\end{proof}
\begin{rem}\label{rem:cycm}
We have an equivalence of $(G,\fcyc)$-spaces
\[
\colim_m \cE(G,\cyc_m)\iso \cE(G,\fcyc)
\]
where the colimit is taken with respect to the partial order of divisibility. Hence for every $\org$-spectrum $E$,  
\[
H_*^G(\cE(G,\fcyc),E)=\colim_mH_*^G(\cE(G,\cyc_m),E).
\]
Moreover it is clear from the proof of Proposition \ref{prop:compuhgctc} that for every $m$ the map
\[
H_*^G(\cE(G,\cyc_m),\CT(\C))\otimes\Q\to H_*^G(\cE(G,\fcyc),\CT(\C))\otimes\Q  
\]
is the inclusion 
\begin{multline}
\bigoplus_{p+q=n}\bigoplus_{(C)\in(\cyc_m)}H_{p}(Z_GC,\Q)\otimes_{\Q[W_GC]}\theta_C\cdot \CT_q(\C[C])\otimes\Q\hookrightarrow\\ \bigoplus_{p+q=n}\bigoplus_{(C)\in(\fcyc)}H_{p}(Z_GC,\Q)\otimes_{\Q[W_GC]}\theta_C\cdot \CT_q(\C[C])\otimes\Q.
\nonumber
\end{multline}

\end{rem}

\section{Conditions equivalent to the rational injectivity of the \texorpdfstring{$\CT$}{} assembly map}\label{sec:equi}

Let $G$ be a group. As shown in the proof of Proposition \eqref{prop:compuhgctc}, we have a direct sum decomposition
\begin{multline*}
H^G_n(\cE(G,\fcyc),HC(\C/\C))=\\
\bigoplus_{p+q=n}\bigoplus_{(C)\in (\fcyc)}H_p(Z_GC,\Q)\otimes_{\Q[W_G(C)]}\theta_C\cdot HC_q(\C[C]/\C).
\end{multline*}
By the same proof, for each $p,q$ we have
\begin{multline}\label{formu:equihc}
\bigoplus_{(C)\in (\fcyc)}H_p(Z_GC,\Q)\otimes_{\Q[W_G(C)]}\theta_C\cdot HC_q(\C[C]/\C)=\\
\bigoplus_{(g)\in\con_f(G)}H_p(Z_G\langle g\rangle,\Q)\otimes HC_q(\C/\C).
\end{multline}
On the other hand we also have a decomposition
\[
HC_n(\C[G]/\C)=\bigoplus_{(g)\in\con(G)}HC_n^{(g)}(\C[G]/\C).
\]
The assembly map identifies
\[
H_n^G(\cE(G,\fcyc),HC(\C/\C))=\bigoplus_{(g)\in\con_f(G)}HC_n^{(g)}(\C[G]/\C).
\] 
Thus there is a projection 
\[
\pi_n^{\fcyc}:HC_n(\C[G]/\C)\to H_n^G(\cE(G,\fcyc),HC(\C/\C))
\]
which is left inverse to the asembly map. By composing $\Tr Sch:KH_{n+1}(\cL^1[G])\to HC_{n-1}(\C[G]/\C)$ with the projection above, we obtain a map
\begin{equation}\label{map:cg}
\pi_{n-1}^{\fcyc}\Tr Sch:KH_{n+1}(\cL^1[G])\otimes\Q\to H_{n-1}^G(\cE(G,\fcyc),HC(\C/\C)).
\end{equation}
Next, if $m\ge 1$ then
\[
H_n^G(\cE(G,\cyc_m),HC(\C/\C))=\bigoplus_{(g)\in\con_f(G), g^m=1}HC_n^{(g)}(\C[G]/\C).
\]
Thus we also have a map
\begin{equation}\label{map:cgm}
\pi_{n-1}^{\cyc_m}\Tr Sch:KH_{n+1}(\cL^1[G])\otimes\Q\to H_{n-1}^G(\cE(G,\cyc_m),HC(\C/\C)).
\end{equation}

In the following proposition we use the following notation. We write

\begin{multline}\label{equiplu}
H_n^G(\cE(G,\cyc_m),\CT(\C)\otimes\Q)^+:=\\
\bigoplus_{p+q=n, q\ge 1}\bigoplus_{(C)\in(\cyc_m)}H_{p}(Z_GC,\Q)\otimes_{\Q[W_GC]}\theta_C\cdot \CT_q(\C[C])\otimes\Q
\end{multline}

and 

\begin{multline}\label{equimin}
H_n^G(\cE(G,\cyc_m),\CT(\C)\otimes\Q)^-:=\\
\bigoplus_{p+q=n, q\le 0}\bigoplus_{(C)\in(\cyc_m)}H_{p}(Z_GC,\Q)\otimes_{\Q[W_GC]}\theta_C\cdot \CT_q(\C[C])\otimes\Q.
\end{multline}

Note that, by Proposition \ref{prop:compuhgctc}, 
$H_n^G(\cE(G,\cyc_m),\CT(\C))\otimes\Q$ is the direct sum of \eqref{equiplu} and \eqref{equimin}.
 
\begin{prop}\label{prop:equiva}
Let $G$ be a group, $n\in\Z$ and $m\ge 1$. The following are equivalent.
\item[i)] The rational assembly map
\begin{equation}\label{map:assCT}
H_n^G(\cE(G,\cyc_m),\CT(\C))\otimes\Q\to \CT_n(\C[G])\otimes\Q
\end{equation}
is injective.
\item[ii)] The restriction of the rational assembly map to the summand \eqref{equiplu} is inyective.

\item [iii)] The image of the map \eqref{map:cgm} coincides with the image of 
\[
\Tr Sch:H^G_{n+1}(\cE(G,\cyc_m),KH(\cL^1))\otimes\Q\to H_{n-1}^G(\cE(G,\cyc_m),HC(\C/\C))
\]
\end{prop}
\begin{proof}
It is clear that i) implies ii). Assume that ii) holds and consider the following commutative diagram with exact columns:
\[
\xymatrix{
H^G_{n+1}(\cE(G,\cyc_m), KH(\cL^1))\otimes\Q\ar[d]^{{\Tr} Sch}\ar[r]& KH_{n+1}(\cL^1[G])\otimes\Q\ar[d]\ar[dl]^{\eqref{map:cgm}}
\\
H^G_{n-1}(\cE(G,\cyc_m),HC(\C/\C))\ar[r]\ar[d]&HC_{n-1}(\C[G])\ar[d]\\
H^G_{n}(\cE(G,\cyc_m),\CT(\C))\otimes\Q\ar[d]\ar[r]&\CT_n(\C[G])\otimes\Q\ar[d]\\
H^G_{n}(\cE(G,\cyc_m),KH(\cL^1))\otimes\Q\ar[r]& KH_n(\cL^1[G])\otimes\Q.\\
}
\]
Let $x$ be an element of the
kernel of the map of part i), that is of the first map above bottom in the diagram above. Write $x=x_++x_-$, with $x_+$ in \eqref{equiplu} and $x_-$ in \eqref{equimin}. The image of $x$ under the vertical map must be zero, since by Yu's theorem (\cite{yu}, see also \cite{cortar}), the bottom horizontal map is injective.  By \eqref{ctlow} and the proof of Proposition \ref{prop:compuhgctc}, this implies that $x_-=0$, proving that ii) implies i). Next assume $y$ is an element in the image of \eqref{map:cgm}
which is not in the image of the vertical map $\Tr Sch$ in the diagram above. Then the image of $y$ under the vertical map is a nonzero element of the kernel
of the next horizontal map. Thus i) implies iii). The converse is also clear, using Yu's theorem again. 
\end{proof}

\begin{coro}\label{coro:fjct}
Let $G$ be a group and let $n,p\in\Z$ with $n\equiv p+1\mod 2$. Assume that the map
\begin{equation}\label{map:fjct}
H^G_p(\cE(G,\fcyc),KH(\cL^1))\otimes\Q\to KH_p(\cL^1[G])\otimes\Q
\end{equation}
is surjective. Then the map \eqref{map:assCT} is injective for every $m\ge 1$.
\end{coro}
\begin{proof}
By Yu's theorem (\cite{yu},\cite{cortar}) the map \eqref{map:fjct} is always injective; under our current assumptions, it is an isomorphism. Moreover, by \cite{CT}*{Theorem 6.5.3}, the groups $KH_p(\cL^1[G])$ depend only on the parity of $p$. It follows that condition iii) of Proposition \ref{prop:equiva} holds for every $m$ and every $n\equiv p+1\mod 2$. This concludes the proof. 
\end{proof}

\section{Rational injectivity of the equivariant regulators}\label{sec:regeq}

Let $G$ be a group. By composing the equivariant character \eqref{map:equitau} with the map induced by the inclusion $\Z\subset \C$ we obtain a map  
\begin{equation}\label{map:equiregZ}
H_*^G( \cE(G,\{1\}),K(\Z))\to H_*^G(\cE(G,\fin),K(\C))\overset{\tau}\to H_*^G(\cE(G,\fin),\CT(\C)).
\end{equation}

\begin{prop}\label{prop:monoequiregZ}
The map \eqref{map:equiregZ} is rationally injective.
\end{prop}
\begin{proof}
We have
\begin{equation}\label{equikz}
H_n^G(\cE(G,\{1\}),K(\Z))\otimes\Q=\bigoplus_{p+q=n}H_p(G,\Q)\otimes K_q(\Z)\otimes\Q.
\end{equation}
By Theorem \ref{thm:kar}, the regulators $K_q(\Z)\to K_q(\C)\to \CT_q(\C)=MK_q(\C)$ induce a monomorphism from \eqref{equikz} to
\begin{equation}\label{hegct}
\bigoplus_{p+q=n}H_p(G,\Q)\otimes \CT_q(\C)\otimes\Q.
\end{equation}
The map \eqref{map:equiregZ} tensored with $\Q$ is the composite of the above monomorphism with the inclusion of \eqref{hegct} as a direct summand in \eqref{hctcrat}.
\end{proof}
 
Let $F$ be a number field, $G$ a group and $m\ge 1$. Let $\zeta_m$ be a primitive $m^{th}$ root
of $1$. The map $\cE(G,\cyc_m)\to \cE(G,\fcyc)$, together with the inclusion 
\[
F\subset F(\zeta_m)\overset\iota\to \C^{\hom(F(\zeta_m),\C)} 
\]
and the character $\tau:K(\C)\to \CT(\C)$, induce a homomorphism
\begin{equation}\label{map:equiregFm}
H^G_*(\cE(G,\cyc_m), K(F))\to H_*^G(\cE(G,\fcyc),\CT(\C))^{\hom(F(\zeta_m),\C)}. 
\end{equation}

\begin{prop}\label{prop:monoequiregF}
The map \eqref{map:equiregFm} is rationally injective.
\end{prop}
\begin{proof}
By \cite{luch}*{Theorem 0.3}, we have
\begin{equation}\label{hequikfcycn}
H^G_n(\cE(G,\cyc_m), K(F))=\bigoplus_{p+q=n}\bigoplus_{(C)\in(\cyc_m)}H_p(Z_GC,\Q)\otimes_{\Q[Z_GC]}\theta_C\cdot K_q(F[C])\otimes\Q.
\end{equation}
By Lemma \ref{lem:regcyc} the maps $K_q(F[C])\to \CT_q(\C[C])^{\hom(F(\zeta_m),\C)}$ with $C\in \cyc_m$ induce a rational 
monomorphism from \eqref{hequikfcycn} to 
\begin{equation}\label{hequictcycn}
\bigoplus_{p+q=n}\bigoplus_{(C)\in(\cyc_m)}H_p(Z_GC,\Q)\otimes_{\Q[Z_GC]}\theta_C\cdot \CT_q(\C[C])^{\hom(F(\zeta_m),\C)}.
\end{equation}
The map \eqref{map:equiregFm} tensored with $\Q$ is the composite of the above monomorphism with the inclusion of \eqref{hequictcycn} as a summand in \eqref{hctcrat}. 
\end{proof}

\section{Comparing conjectures and assembly maps}\label{sec:compa}
\numberwithin{equation}{section}
\begin{thm}\label{thm:forz}
Let $G$ be a group. 
Assume that the equivalent conditions of Proposition \ref{prop:equiva} hold for $G$ 
with $m=1$. Then the assembly map
\[
H^G_n(\cE(G,\{1\}),K(\Z))\to K_n(\Z[G])
\]
is rationally injective. In particular this is the case whenever $G$ satisfies the rational 
$KH$-isomorphism conjecture with $\cL^1$-coefficients.
\end{thm}
\begin{proof}
Immediate from Proposition \ref{prop:monoequiregZ} and Corollary \ref{coro:fjct}. 
\end{proof}

\begin{thm}\label{thm:forf}
Let $G$ be a group and $m\ge 1$. Assume that the equivalent conditions of Proposition \ref{prop:equiva} 
hold for $G$ and $m$. Then for every number field $F$, the assembly map
\[
 H^G_n(\cE(G,\cyc_m),K(F))\to K_n(F[G])
\]
is rationally injective. If moreover the condition holds for all $m$ --as is the case, for example, if $G$ satisfies the rational $KH$-isomorphism conjecture with $\cL^1$-coefficients- then $G$ satisfies the rational injectivity part of the $K$-theory isomorphism conjecture with coefficients in any number 
field.
\end{thm}
\begin{proof}
Immediate from Proposition \ref{prop:monoequiregF} and Corollary \ref{coro:fjct}.
\end{proof}


\begin{bibdiv}
\begin{biblist}

\bib{balukh}{article}{
   author={Bartels, Arthur},
   author={L{\"u}ck, Wolfgang},
   title={Isomorphism conjecture for homotopy $K$-theory and groups acting
   on trees},
   journal={J. Pure Appl. Algebra},
   volume={205},
   date={2006},
   number={3},
   pages={660--696},
}

\bib{friendly}{article}{
   author={Corti{\~n}as, Guillermo},
   title={Algebraic v. topological $K$-theory: a friendly match},
   conference={
      title={Topics in algebraic and topological $K$-theory},
   },
   book={
      series={Lecture Notes in Math.},
      volume={2008},
      publisher={Springer},
      place={Berlin},
   },
   date={2011},
   pages={103--165},
}

\bib{corel}{article}{
author={Corti{\~n}as, Guillermo},
author={Ellis, Eugenia},
title={Isomorphism conjectures with proper coefficients},
eprint={arXiv:1108.5196v3},
}
\bib{cortar}{article}{
author={Corti{\~n}as, Guillermo},
author={Tartaglia, Gisela},
title={Operator ideals and assembly maps in K-theory.},
journal={Proc. Amer. Math. Soc. to appear.},
}
\bib{CT}{article}{
   author={Corti{\~n}as, Guillermo},
   author={Thom, Andreas},
   title={Comparison between algebraic and topological $K$-theory of locally
   convex algebras},
   journal={Adv. Math.},
   volume={218},
   date={2008},
   number={1},
   pages={266--307},
}

\bib{cq}{article}{
   author={Cuntz, Joachim},
   author={Quillen, Daniel},
   title={Excision in bivariant periodic cyclic cohomology},
   journal={Invent. Math.},
   volume={127},
   date={1997},
   number={1},
   pages={67--98},
}

\bib{karast}{article}{
   author={Karoubi, Max},
   title={Homologie cyclique et $K$-th\'eorie},
   language={French, with English summary},
   journal={Ast\'erisque},
   number={149},
   date={1987},
   pages={147},
}

\bib{karmult}{article}{
   author={Karoubi, Max},
   title={Sur la $K$-th\'eorie multiplicative},
   language={French, with English summary},
   conference={
      title={Cyclic cohomology and noncommutative geometry (Waterloo, ON,
      1995)},
   },
   book={
      series={Fields Inst. Commun.},
      volume={17},
      publisher={Amer. Math. Soc.},
      place={Providence, RI},
   },
   date={1997},
   pages={59--77},
}

\bib{luch}{article}{
   author={L{\"u}ck, Wolfgang},
   title={Chern characters for proper equivariant homology theories and
   applications to $K$- and $L$-theory},
   journal={J. Reine Angew. Math.},
   volume={543},
   date={2002},
   pages={193--234},
}

\bib{lurela}{article}{
   author={L{\"u}ck, Wolfgang},
   title={The relation between the Baum-Connes conjecture and the trace
   conjecture},
   journal={Invent. Math.},
   volume={149},
   date={2002},
   number={1},
   pages={123--152},
}

\bib{LR2}{article}{
   author={L{\"u}ck, Wolfgang},
   author={Reich, Holger},
   title={The Baum-Connes and the Farrell-Jones conjectures in $K$- and
   $L$-theory},
   conference={
      title={Handbook of $K$-theory. Vol. 1, 2},
   },
   book={
      publisher={Springer},
      place={Berlin},
   },
   date={2005},
   pages={703--842},
}

\bib{LR1}{article}{
   author={L{\"u}ck, Wolfgang},
   author={Reich, Holger},
   title={Detecting $K$-theory by cyclic homology},
   journal={Proc. London Math. Soc. (3)},
   volume={93},
   date={2006},
   number={3},
   pages={593--634},
}

\bib{kh}{article}{
   author={Weibel, Charles A.},
   title={Homotopy algebraic $K$-theory},
   conference={
      title={Algebraic $K$-theory and algebraic number theory (Honolulu, HI,
      1987)},
   },
   book={
      series={Contemp. Math.},
      volume={83},
      publisher={Amer. Math. Soc.},
      place={Providence, RI},
   },
   date={1989},
   pages={461--488},
}


\bib{yu}{article}{
author={Yu, Guoliang},
title={The algebraic K-theory Novikov conjecture for group algebras over the ring of Schatten class operators},
eprint={arXiv:1106.3796},
}

\end{biblist}
\end{bibdiv}
\end{document}